\documentclass[12pt,a4paper,oneside]{amsart}
\usepackage{amsfonts, amsmath, amssymb, amsthm}
\usepackage[colorlinks=true,citecolor=blue]{hyperref}
\usepackage[scale = .7]{geometry}

\usepackage{txfonts}
\usepackage[T1]{fontenc}

\usepackage{graphicx}
\usepackage{parskip}
\usepackage{enumerate}
\usepackage{verbatim}
\usepackage{tikz}
\usepackage{units}

\newtheorem{theorem}{Theorem}
\newtheorem*{theorem*}{Theorem}
\newtheorem{lemma}[theorem]{Lemma}

\newtheorem{claim}[theorem]{Claim}
\newtheorem{obs}[theorem]{Observation}
\newtheorem*{claim*}{Claim}
\newtheorem*{obs*}{Observation}
\newtheorem*{lemma*}{Lemma}
\newtheorem*{sprm*}{Question}

\theoremstyle{definition}

\newtheorem*{remark*}{Remark}

\newcommand{\df}[1]{{{\color{black}\em #1}}}
\newcommand{\mb}[1]{\mathbb{#1}}

\newcommand{\es}{\mb{S}^2}
\newcommand{\ess}[1]{S_{\hspace{-.3ex}{#1}}}

\newcommand{\RR}{\mathbb{R}}

\newcommand{\rpm}{\raisebox{.25ex}{\small $\pm$}}

\begin{document} 

\title{Orthogonal colorings of the sphere}  

\author{Andreas F. Holmsen}
\address{Andreas F. Holmsen \\ Department of Mathematical Sciences \\ KAIST \\
  Daejeon \\ South Korea} 
\email{andreash@kaist.edu}

\author{Seunghun Lee}
\address{Senghun Lee \\ Department of Mathematical Sciences \\ KAIST \\
  Daejeon \\ South Korea} 
\email{prosolver@kaist.ac.kr}

\maketitle 

\begin{abstract} 
An orthogonal coloring of the two-dimensional unit sphere $\es$, is a partition of $\es$ into parts such that no part contains a pair of orthogonal  points, that is, a pair of points at spherical distance $\pi/2$ apart. It is a well-known result that an orthogonal coloring of $\es$ requires at least four parts, and orthogonal colorings with exactly four parts can easily be constructed from a regular octahedron centered at the origin. An intriguing question is whether or not every orthogonal 4-coloring of $\es$ is such an octahedral coloring. In this paper we address this question and show that if every color class has a non-empty interior, then the coloring is octahedral. Some related results are also given.
\end{abstract}

\section{Introduction} \label{intro}

Let $\es$ denote the unit sphere centered at the origin in $\RR^3$. A partition of $\es$ into $k$ non-empty parts is called a \df{$k$-coloring} and we refer to the parts as \df{color classes}. We call a $k$-coloring, $\es = \ess 1 \cup \cdots \cup \ess k$, \df{orthogonal} if $u \cdot v \neq 0$ for any pair $u$ and $v$ which belong to the same color class $\ess i$. Here $u \cdot v$ denotes the standard dot product in $\RR^3$. 

It is a well-known result that an orthogonal coloring of $\es$ requires at least {\em four} parts\footnote{Further remarks and historical notes will be given at the end of this section}, and a typical orthogonal 4-coloring of $\es$ can be obtained as follows. Define subsets
\[\def\arraystretch{1.4}
\begin{array}[h]{cccccc}
A_1 & \coloneqq & \{(x,y,z) \in \es \: : &  x> 0, & y\geq 0, &  z \geq 0 \}, \\
A_2 & \coloneqq & \{(x,y,z) \in \es \: : & x\leq 0, & y> 0, & z\geq 0\}, \\
A_3 &  \coloneqq & \{(x,y,z) \in \es \: : & x\leq 0,&  y\leq 0, & z >0\}, \\
A_4 &  \coloneqq & \{(x,y,z) \in \es \: : & x> 0,  &y < 0, & z > 0\}, 
\end{array}
\] and set $\ess i \coloneqq  A_i \cup (-A_i)$. It is easily seen that this gives us an orthogonal 4-coloring of $\es$. Moreover, this 4-coloring can be modified in (uncountably) many ways by carefully moving some boundary points from one color class to another, but we would like to consider all of these colorings to be ``equivalent'' in some sense. 

A \df{regular octahedron} is the convex hull of the set of vectors $\{\rpm u_1,\rpm u_2,\rpm u_3\}$, where $u_1, u_2, u_3$ is an orthonormal basis for $\RR^3$. Given a subset $S\subset \es$, let $\ell(S)$ denote the set of lines which  pass through the origin and a point from $S$. We call the set $S$ \df{octahedral} if all lines in $\ell(S)$ pass through the closure of two opposite faces of a  regular octahedron. We can now give a characterization of the type of 4-colorings described in the previous paragraph: A 4-coloring of $\es$ is \df{octahedral} if and only if each color class is octahedral. Notice that an octahedral coloring is not necessarily orthogonal, and that in any octahedral coloring, each color class is octahedral with respect to the same fixed regular octahedron.

Now the following natural and intriguing question presents itself.

\begin{sprm*}
Does there exist an orthogonal 4-coloring of $\:\es$ which is not octahedral ?  
\end{sprm*}

We are not able to answer this question, but we find it believable that such colorings exist. The goal of this paper is to give some alternative characterizations of orthogonal 4-colorings of $\es$ which {\em are} octahedral. These will be in terms of certain local conditions on the color classes, and may therefore provide further insights into orthogonal 4-colorings of $\es$.

Our first condition is a topological one. Let $S$ be an arbitrary subset of $\es$. We say that a non-empty open subset $D\subset \es$ is a \df{dominating set} of $S$, or that $D$ dominates $S$, if every great circle in $\es$ which intersects $D$ also intersects $S$. Notice that if $D$ dominates $S$, then $-D$ also dominates $S$. Notice also that if $D$ is an open set contained in $S$, then $D$ dominates $S$, but the converse does not hold in general. In particular, a zero-measure set $S$ can have a dominating set $D$, which is not contained in $S$ since $D$ must have positive measure. We may now state our first result.

\begin{theorem}\label{domination}
  An orthogonal 4-coloring of $\:\es$ is octahedral if and only if every color class $\ess i$ has a dominating set $D_i$.  
\end{theorem}

Our second condition is a combinatorial one. Let $S$ be an arbitrary subset of $\es$. We say that $S$ is \df{locally octahedral} if every subset $S'\subset S$ such that $|S'| \leq 3$ is octahedral. Equivalently, this means that $(u \cdot v) (u \cdot  w) (v \cdot  w) \geq 0$ for any $u,v,w \in S$. We have the following.

\begin{theorem}\label{weak}
An orthogonal 4-coloring of $\:\es$ is octahedral if and only if each color class is locally octahedral.  
\end{theorem}

The proof of Theorem \ref{domination} uses elementary methods and is given in Section \ref{P1}. Our proof of Theorem \ref{weak} is given in Section \ref{P2}  and relies on Theorem \ref{domination}. We conclude with some final remarks in Section \ref{C}.

\bigskip

{\bf {\em Remarks and historical notes.}}
The study of orthogonal colorings of $\es$ implicitly dates back to the 1960's and the celebrated theorem  of Kochen and Specker \cite{kochen}. They presented a finite subconfiguration $X\subset \es$ which is {\em KS-uncolorable}, which in particular implies that  there are no orthogonal 3-colorings of $\es$. Kochen and Specker's original configuration consisted of 117 vectors, and more recently Peres \cite{peres} found a configuration of 33 vectors. The following configuration of 13 vectors was communicated to us by Evan DeCorte \cite{edc}, and the reader is invited to check that it results in a graph with chromatic number 4. 
\[
\begin{array}{ccccc}
  (1,0,0) \; ,  & & (0,1,0) \; ,  &  & (0,0,1) \; ,   \\ 
  (0,1,1) \; , & & (1,0,1) \; , &  & (1,1,0)  \; , \\
   (0,1,-1) \; , &  & (1,0,-1) \; ,   &  & (1,-1,0) \; ,  \\
   (-1,1,1) \; ,  &  & (1,-1,1) \; , &  & (1,1,-1)  \; , \\
& & (1,1,1) \; . & &
\end{array}
\]

Around the same time as the Kochen--Specker theorem was discovered, Erd{\H o}s made several conjectures regarding colorings of $\es$ (and other spaces as well). An elementary proof that an orthogonal coloring of $\es$ requires at least 4 colors was given by Simmons \cite{simmons} where an explicit octahedral coloring is also described. 

From another point of view, Hales and Straus \cite{hales} studied colorings of the projective plane $\mathbb{F}P^2$ in which any line contains at most two distinct colors. They showed that such colorings are in one-to-one correspondence with the non-Archimedean valuations on $\mathbb{F}$. By setting $\mathbb{F} = \mathbb{R}$, Godsil and Zaks \cite{godsil} used this result to give a different proof of the fact that $\es$ does not admit an orthogonal 3-coloring. In addition, they showed that the 2-adic valuation on $\mathbb{Q}$ produces an orthogonal 3-coloring of the {\em rational unit sphere}, $\es \cap {\mathbb{Q}^3}$, in which each color class is dense in $\es$.

Colorings of the $d$-dimensional unit sphere $\mathbb{S}^d$ have also been considered for $d>2$. It was shown by Lov{\'a}sz \cite{lov}, using topological techniques, that for any $-1 < \alpha < 1$ and any $d$-coloring of $\mathbb{S}^d$ there is a color class that contains a pair of vectors $u$ and $v$ such that $u\cdot v = \alpha$. This result was significantly improved by Raigorodskii \cite{raig} who showed, using linear algebra techniques, that the same conclusion holds for $k$-colorings of $\mathbb{S}^d$ where $k=k(d)$ grows exponentially.
  
The problems on orthogonal colorings of $\es$ (and $\mb{S}^d$) belong to a larger class of questions concerning colorings of graphs arising from metric spaces. Among these,  the arguably most famous one is the Hadwiger--Nelson problem which asks for the minimum number of colors needed to color the points of $\RR^2$ such that no two points at distance 1 apart have the same color. It is known that this number is at least 4 and at most 7, but the exact value may depend on the choice of axioms for set theory, as was demonstrated by Shelah and Soifer \cite{soifer1, soifer2}. Several variations of this problem have been considered, and by imposing additional conditions on the color classes the lower bound can be improved. For instance, Falconer \cite{falconer} showed that if every color class is Lebesgue measurable, then at least 5 colors are needed, while Woodall \cite{woodall} showed that if the color classes are bounded by Jordan curves, then at least 6 colors are needed. For more information about recent results, developments, and further directions of research we refer the reader to Raigorodskii's survey articles \cite{raiS1, raiS2} as well as Soifer's colorful book \cite{soifs}.


\section{Proof of Theorem \ref{domination}}\label{P1}

Before we get to the proof let us recall the notion of spherical convexity. Any pair of non-antipodal points $u$ and $v$ in $\es$ are connected by a unique \df{geodesic arc} contained in a unique \df{great circle} passing through $u$ and $v$. This allows us to define convexity for subsets of $\es$ which are contained in an open hemisphere $H\subset \es$. In fact, central projection into an affine plane parallel to the boundary of $H$ provides a homeomorphism between $H$ and $\RR^2$ which preserves convexity. Thus a set $S\subset \es$ is convex if and only if it does not contain antipodal points and it contains, with each pair of its points, the small arc of the great circle determined by them. {(The reader may consult for instance \cite[Section 9]{dgk} for more about spherical convexity, where our type of convex sets are called {\em strongly convex}.)} 

\bigskip
We now start the proof of Theorem \ref{domination}. Let $\es = \ess 1 \cup \ess 2 \cup \ess 3 \cup \ess 4$ be an orthogonal coloring where each color class has a dominating set. Our goal is to show that the coloring is octahedral. (The other direction is trivial.) Here is an obvious observation which will be used several times in our argument.

\begin{obs}\label{4c}
There is no great circle which intersects every color class.
\end{obs}

\subsection{The connected components of dominating sets}
Let $R_k$ denote the union of all sets $D_k$ which dominate $\ess k$, that is, the inclusion-maximal set which dominates $\ess k$. By our hypothesis and the definition of dominating sets, $R_k$ is a non-empty open set. The first basic observation is the following.
\begin{claim}
Every connected component of $R_k$ is an open convex set. 
\end{claim}
\begin{proof} Notice that if $x$ and $y$ are distinct points which belong to the same connected component of $R_k$, then every great circle which separates $x$ from $y$ (that is, $x$ and $y$ are contained in opposite open hemispheres) must contain a point in $\ess k$.  In particular, $x$ and $y$ cannot be a pair of antipodal points: If this were the case, then {\em every} great circle must contain a point in $\ess k$, and it is easily seen that such a coloring can not be orthogonal. 

We may therefore assume there is a unique geodesic arc $\alpha$ which connects the points $x$ and $y$. For a point $p \in \alpha$, every great circle which passes through $p$ must contain a point in $\ess k$, and it is easily seen that this also holds for any point in a sufficiently small $\varepsilon$-neighborhood of $p$. Thus $p$ is contained in a dominating set of $\ess k$, and consequently the entire segment $\alpha$ belongs to a dominating set of $\ess k$ which obviously belongs to the same connected component of $R_k$ as $x$ and $y$. \end{proof}

\subsection{The boundaries of dominating sets} For every $1 \leq k \leq 4$, consider a fixed connected component of $R_k$ and let $C_k$ denote its boundary. This gives us four convex curves $C_1$, $C_2$, $C_3$, and $C_4$. By definition, every great circle which intersects $C_k$ transversally contains a point of $\ess k$. The connected component of $R_k$ bounded by $C_k$ will be referred to as {\em the bounded region} of $C_k$.

We say that a point $p\in C_k$ has {\sc Type $j$} if every $\varepsilon$-neighborhood of $p$ contains a point $q\neq p$ such that $q \in\ess j$. Note that in general a point in $C_k$ can have more than one {\sc Type}, and since $C_k$ is the boundary of a connected component of $R_k$ we have the following.
\begin{obs} \label{j not k}
For every point $p\in C_k$, there exists a $j\neq k$ such that $p$ has {\sc Type $j$}. 
\end{obs}

For every $j \neq k$ let $T_k(j)\subset C_k$ denote the set of points in $C_k$ of {\sc Type $j$}. By Observation \ref{j not k} we have $C_k = \bigcup_{j\neq k}T_k(j)$, and it is easily seen that each $T_k(j)$ is a closed subset of $C_k$. Therefore the following holds.

\begin{obs}\label{open arcs}
For any open arc $A\subset C_k$, there exists a $j\neq k$ such that the intersection $A\cap T_k(j)$ contains an open arc.
\end{obs}

We are now ready for the key step of the proof.

\begin{lemma} \label{crucial 1}
For every $j\neq k$, the set $T_k(j)\subset C_k$ has non-empty interior.
\end{lemma}
\begin{proof}
 For the sake of argument let us set $k=4$, and suppose, contrary to the claim, that the interior of $T_4(3)$ is empty. We split the argument into two cases.

{\em First case.} Suppose both $T_4(1)$ and $T_4(2)$ both have non-empty interiors. If there exists a point $p$ belonging to the bounded regions of $C_3$ and $C_4$, then we reach a contradiction as follows: Since $T_4(1)$ and $T_4(2)$ are both non-empty closed sets which cover $C_4$, there exists a point $q\in T_4(1)\cap T_4(2)$. The great circle which passes through $p$ and $q$, intersects $C_4$ in another point $r\neq q$ which we may assume has {\sc Type} 2. Thus there are points $q' \in \ess 1$ and $r' \in \ess 2$ sufficiently close to $q$ and $r$, respectively, such that the great circle which passes through $q'$ and $r'$ will intersect both $C_3$ and $C_4$ transversally. However, this great circle would intersect every color class, which contradicts Observation \ref{4c}. 

We now assume the bounded regions of $C_3$ and $C_4$ are disjoint. Choose a point $p$ in the bounded region of $C_3$ such that the two great circles, $\gamma$ and $\gamma'$, which pass through $p$ and are tangent to $C_4$, each intersect $C_4$ in single points, say, $x$ and $x'$. Notice that the set of exceptional points where this property fails is contained in an  at most countable union of great circles. Thus the set of exceptional points has measure zero, so the existence of the point $p$ is guaranteed. The set $C_4\setminus \{x,x'\}$ consists of two disjoint open arcs, and any great circle through $p$ which intersects one of these arcs also intersects the other, and thus intersects $C_4$ transversally. If any of these two arcs contain a point of {\sc Type} 1 and {\sc Type} 2, then we may proceed as in the previous paragraph. The only other alternative is that one of the arcs consists only of points of {\sc Type} 1, while the other arc consists only of points of {\sc Type} 2, and again we may apply the argument from the previous paragraph. We can therefore conclude that it is impossible that both $T_4(1)$ {\em and} $T_4(2)$ have non-empty interiors (under the assumption that $T_4(3)$ has empty interior).
 
{\em Second case.} Suppose $T_4(2)$ has empty interior, which implies that $C_4 = T_4(1)$.  Choose a point $p$ which is orthogonal to both a point in the bounded region of $C_1$ {\em and} a point in the bounded region of $C_4$. Let $D$ be a sufficiently small open neighborhood of $p$ consisting of points which are simultaneously orthogonal to a point in the bounded region of $C_1$ {\em and} a point in the bounded region of $C_4$. Note that $D\subset \ess 2 \cup \ess 3$ since the orthogonal complement of a point in $D$ is a great circle which intersects $C_1$ and $C_4$ transversally. Without loss of generality we may assume that the point $p \in D$ belongs to  $\ess 2$, and we now show that this forces $D$ to contain a subset which dominates $\ess 2$. 

To see why this is true, consider a great circle $\gamma$, which passes through $p$ and intersect $C_4$ transversally. By our assumption, $\gamma$ intersects $C_4$ in a point $q$ of {\sc Type} 1, and for any point $q'$ sufficiently close to $q$, the great circle passing through $p$ and $q'$ will intersect $C_4$ transversally. Therefore there exists a great circle $\gamma'$ passing through $p$ which contains points in $\ess 1$, $\ess 2$, and $\ess 4$. By Observation \ref{4c}, it follows that the open geodesic arc $\gamma'\cap D$ is contained in $\ess 2$. Applying this same argument to another great circle which passes through $p$ and intersects $C_4$ transversally, results in another open geodesic arc $\gamma''\cap D \subset \ess 2$. These two geodesic arcs cross at the point $p$, and the convex hull their union dominates $\ess 2$.

Without loss of generality we can assume that $C_2$ is the boundary of the connected component of $R_2$ which contains the point $p$. If $T_2(3)$ has non-empty interior, then we can find a point $q\in C_2$ of {\sc Type} 3, and a point $r\in C_4$ of {\sc Type 1}, such that the great circle passing through $q$ and $r$ intersects both $C_2$ and $C_4$ transversally. As before, if $q'$ and $r'$ are sufficiently close to $q$ and $r$, respectively, then the great circle through $q'$ and $r'$ will intersect both $C_2$ and $C_4$ transversally, and we could find a great circle which intersects every color class, contradicting Observation \ref{4c}. We may therefore assume that the interior of $T_2(3)$ is empty.

By applying the first case of the argument to the curve $C_2$, it follows that either $T_2(1)$ or $T_2(4)$  has empty interior. In other words, we may assume that either $C_2 = T_2(4)$ or $C_2 = T_2(1)$. Suppose $C_2 = T_2(4)$. Let $q$ denote a point in the bounded region of $C_4$ which is orthogonal to the point $p$ (which is in the bounded region of $C_2$). The orthogonal complement of $q$ is a great circle $\gamma$ which passes through $p$, and since $p$ is in the bounded region of $C_2$, it follows that $\gamma$ intersects $C_2$ transversally. Let $r$ be one of the intersection points between $\gamma$ and $C_2$. If a point $r'$ is sufficiently close to $r$, then its orthogonal complement is a great circle which intersects $C_4$ transversally, and therefore contains a point in $\ess 4$. However, since $r$ has {\sc Type} 4, this would result in a pair of orthogonal points in $\ess 4$.

The situation is symmetric if we suppose that $C_2 = T_2(1)$. In this case, we choose a point $q$ in the inner region of $C_1$ which is orthogonal to the point $p$. We may now proceed as in the previous paragraph, resulting in a pair of orthogonal points in $\ess 1$.\end{proof}

\subsection{The shape of dominating sets}
Let $V_k(j)$ denote the interior of $T_k(j) \subset C_k$. By Lemma \ref{crucial 1} we know that $V_k(j)\neq \emptyset$ for all $j\neq k$, and the following claim implies that all the sets $V_k(j)$ are contained in the union of three great circles. 
\begin{claim} \label{all permutations}
For any permutation $\pi = (i,j,k,l)$ of $\{1,2,3,4\}$, there exists a great circle which contains the sets $V_i(j)$ and $V_k(l)$.
\end{claim}
\begin{proof} Note that it suffices to prove the claim for a single pair of open connected arcs $A\subset V_i(j)$ and $B \subset V_k(l)$. If $A\cup B$ is not contained in a great circle, then there must exist a great circle $\gamma$, which intersects $A$ and $B$ transversally. This implies that close to $\gamma$ we can find a great circle $\gamma'$ which contains points in $\ess j$ and $\ess l$ and intersect the arcs $A$ and $B$, and therefore also the curves $C_i$ and $C_k$ transversally. In this case $\gamma'$ intersects every color class, contradicting Observation \ref{4c}.
\end{proof}

Claim \ref{all permutations} imposes severe restrictions on the shape of the curves $C_k$, because Observation \ref{open arcs} implies that every open interval of $C_k$ is contained in one of the great circles described in Claim \ref{all permutations}. Moreover, Claim \ref{all permutations} implies that the sets $V_i(j)$ and $V_j(i)$ belong to the same great circle, and it follows that there are three great circles $\gamma_1$, $\gamma_2$, and $\gamma_3$ whose union contain every curve $C_k$:
\[\def\arraystretch{1.4}
\begin{array}[h]{ccc}
V_1(2) \cup V_3(4) \cup V_2(1)  \cup V_4(3)  & \subset &  \gamma_1, \\ 
V_1(3)  \cup V_2(4)  \cup V_3(1) \cup V_4(2)  & \subset &  \gamma_2, \\
V_1(4) \cup V_2(3) \cup  V_4(1) \cup V_3(2)  & \subset &  \gamma_3. 
\end{array}
\]
We are now ready to complete the proof of Theorem \ref{domination}. Let $n_i\in \es$ be a normal vector to the great circle $\gamma_i$. Our goal is to show that the vectors $n_1$, $n_2$, $n_3$ form an orthonormal basis for $\RR^3$. Let us first show that $n_1$, $n_2$, $n_3$ are linearly independent. The great circles $\gamma_1$, $\gamma_2$, $\gamma_3$ divide $\es$ into regions, and each curve $C_k$ is the boundary of such a region. If the vectors $n_1$, $n_2$, $n_3$ are not linearly independent, then we can find a great circle $\gamma$ which intersects the interior of every region formed by $\gamma_1$, $\gamma_2$, $\gamma_3$. This implies that $\gamma$ intersects each curve $C_k$ transversally, and therefore intersects every color class. We may therefore conclude that $n_1$, $n_2$, $n_3$ are linearly independent, and consequently, the three great circles $\gamma_1$, $\gamma_2$, $\gamma_3$ divide $\es$ into eight triangular regions which come in antipodal pairs $\{ Q_1 , -Q_1 \}$, $\{Q_2,-Q_2\}$, $\{Q_3,-Q_3\}$, and $\{Q_4 ,-Q_4\}$. Notice that for any triple of regions $Q_{i_1}$, $Q_{i_2}$, $Q_{i_3}$ there exists a great circle which intersects the interior of each of the regions. Therefore, since $R_k$ is symmetric about the origin, we may assume (using Observation \ref{4c}, again) that each curve $C_k$ is the boundary of $Q_k$ or $-Q_k$ for $k = 1,2,3,4$. 

It remains to show that $n_i \cdot n_j = 0$ for $i\neq j$. Let $T$ denote the interior of the spherical triangle with vertices $n_1$, $n_2$, $n_3$. The important observation is that every great circle which is orthogonal to a point in $T$ will intersect the same triple of regions, $Q_{i_1}$, $Q_{i_2}$, $Q_{i_3}$, and therefore every point in $T$ must belong to same color class. In other words, $T\subset \ess k$ where $k\neq i_1, i_2, i_3$. Consequently, if a pair of the normal vectors satisfy $n_i \cdot n_j < 0$, then $T$ would contain a pair of orthogonal points in the same color class. This shows that the normal vectors $n_1$, $n_2$, $n_3$ form an orthonormal basis for $\RR^3$, and each color class $\ess k$ contains the interiors of an antipodal pair spherical triangles $Q_k$ and $-Q_{k}$ which have all right angles. In other words, the coloring is octahedral.

\section{Proof of Theorem \ref{weak}} \label{P2}

Clearly, in an octahedral 4-coloring of $\es$, every color class is locally octahedral, so it remains to show the other direction. This will be deduced from the following result which is of independent interest.

\begin{lemma} \label{3uniform}
  For any 3-coloring of $\:\es$ there exists points $u ,v, w$ in the same color class such that  $(u \cdot v) (u \cdot w) (v \cdot w) < 0$.
\end{lemma}

Before getting to the proof, let us show how Lemma \ref{3uniform} implies Theorem \ref{weak}. 

\begin{proof}[Proof of Theorem \ref{weak}]
Consider an orthogonal 4-coloring which is locally octahedral. We want to show that it is an octahedral coloring.  If every color class contains an open set in $\es$, then the coloring is octahedral by Theorem \ref{domination}. So without loss of generality, we assume that $\ess 4$ has empty interior. This implies that $\es = A_1 \cup A_2 \cup A_3$ where $A_i$ is the closure of $\ess i$.

By assumption, for any triple of points $u, v, w \in \ess i$ we have $(u \cdot v) (u \cdot w) (v \cdot w) \geq 0$. The important observation is that this inequality also holds for triples of points $u,v,w \in A_i$. If we set $T_1 = A_1$, $T_2 = A_2\setminus A_1$, and $T_3 = A_3\setminus (A_1\cup A_2)$, we obtain a 3-coloring, $\es = T_1\cup T_2 \cup T_3$, which violates Lemma \ref{3uniform}. \end{proof}

For our proof of Lemma \ref{3uniform} we need the following.

\begin{obs} \label{circle2}
  Suppose the unit circle, $\mathbb{S}^1$, is covered by two closed sets $B_1$ and $B_2$ such that for any $u,v,w \in B_i$ we have $(u \cdot v) (u \cdot w) (v \cdot w) \geq 0$. Then each $B_i$ consists of a pair of closed antipodal arcs of length $\pi/2$, and $B_1\cap B_2$ consists of the vertices of a square inscribed in $\mathbb{S}^1$.
\end{obs}

\begin{proof}
Since $B_1$ and $B_2$ are closed and $\mathbb{S}^1 = B_1 \cup B_2$, it follows that $B_1\cap B_2\neq \emptyset$. Fix a point $p_0 \in B_1\cap B_2$, and for $\theta \geq 0$, let $p_\theta \in \mathbb{S}^1$ denote the point of $\mathbb{S}^1$ whose arc length from $p_0$ equals $\theta$ measured in counter-clockwise direction. Let $X_1$ be the open circular arc $(p_{\pi/2}, p_\pi)$, and let $X_2$ be the open circular arc $(p_{\pi}, p_{3\pi/2})$. For any $\pi/2 < \theta < 3\pi/2$ we have $p_0 \cdot p_{\theta} < 0$, so if we choose points $p_{\theta_1}\in X_1$ and $p_{\theta_2} \in X_2$ such that $\theta_2 - \theta_1 > \pi/2$, then we have $(p_0 \cdot p_{\theta_1}) (p_0 \cdot p_{\theta_2}) (p_{\theta_1} \cdot p_{\theta_2}) < 0$. Since $p_0\in B_1\cap B_2$, the hypothesis implies that $p_{\theta_1}$ and $p_{\theta_2}$ must belong to different $B_i$ so we may assume $p_{\theta_1}\in B_1$ and $p_{\theta_2}\in B_2$. We claim that this implies that $X_1\subset B_1\setminus B_2$ and $X_2\subset B_2\setminus B_1$. To see this, notice that for any $\pi/2 < \theta < \theta_1$, we have $p_\theta \cdot p_{\theta_2} < 0$, and therefore $p_\theta \in B_1 \setminus B_2$. Now, for any $\pi < \theta < 3\pi/2$, there exists an $\varepsilon > 0$ such that $\pi/2 < \theta - \pi/2 - \epsilon < \theta_1$, which shows that $X_2 \subset B_2 \setminus B_1$. The same reasoning shows that $X_1 \subset B_1 \setminus B_2$, and since $B_1$ and $B_2$ are closed it follows that $p_\pi \in B_1 \cap B_2$. Therefore we can apply the previous argument to the open intervals $Y_1 = (p_{3\pi/2},p_{2\pi})$ and $Y_2 = (p_0,p_{\pi/2})$, which implies that one of them is contained in $B_1\setminus B_2$ while the other is contained in $B_2\setminus B_1$. It is however easily seen that we must have $Y_1 \subset B_1\setminus B_2$ and $Y_2 \subset B_2\setminus B_1$, which shows that $B_1 = [\pi/2, \pi] \cup [3\pi/2,2\pi]$ and  $B_2 = [0, \pi/2] \cup [\pi,3\pi/2]$. \end{proof}

\begin{proof}[Proof of Lemma \ref{3uniform}]
Let $\es = \ess 1 \cup \ess 2 \cup \ess 3$ be a 3-coloring and let $A_i$ be the closure of $\ess i$. Note that if there exists points $u,v,w \in A_i$ such that $(u \cdot v) (u \cdot w) (v\cdot w) < 0$, then we can also find points $u',v',w' \in \ess i$ which also satisfy this inequality.

At least one of the $A_i$ must have non-empty interior, and without loss of generality we may assume that the point $(0,0,1)$ is an interior point of $A_3$. Therefore there exists a (small) $\varepsilon>0$ such that every point in an $\varepsilon$-neighborhood of the equatorial great circle $\gamma_0 = \{(x,y,0) \: : \: x^2 + y^2 = 1\}$, denoted by $N_\varepsilon$, is orthogonal to a point in $A_3$. 

Suppose there exists a point $u \in A_3 \cap N_\varepsilon$. By our assumption, there is a point $p$ in the interior of $A_3$ such that $u \cdot p = 0$. If we consider the great circle $\gamma$ passing through $u$ and $p$, it is easily seen that  we can find points $v, w \in \gamma \cap A_3$ which are close to $p$  such that $u \cdot v > 0$, $u\cdot w < 0$, and $v \cdot w > 0$.

Now suppose that $N_\varepsilon \subset A_1\cup A_2$. If $\gamma_0$ does not contain the three points we are looking for, then Observation \ref{circle2} can be applied, and after a suitable rotation we may assume that the points $(\rpm 1,0,0)$ and $(0,\rpm 1, 0)$ are contained in $A_1 \cap A_2$. Moreover, by slightly rotating $\gamma_0$ about the $x$-axis while staying within $N_\varepsilon$, we may refer to Observation \ref{circle2} again to see that the point $u = (0, \cos \alpha ,  \sin \alpha)$ is contained in $A_1\cap A_2$, where $\alpha>0$ is sufficiently small. Similarly, by rotating $\gamma_0$ about the $y$-axis we find that the points $v = (\cos \alpha, 0, \sin \alpha)$ and $w = (\cos \alpha, 0, -\sin \alpha)$ are also contained in $A_1\cap A_2$. Thus we have found points $u,v,w \in A_1\cap A_2$ such that $(u \cdot v) (u \cdot w) (v \cdot w) < 0$.
\end{proof}

\section{Final remarks} \label{C}

\subsection{Noel's conjecture} We want to point out that the condition of Theorem \ref{weak} was originally suggested to us by Noel \cite{noel} under the name {\em weakly octahedral}. In fact, Noel \cite{noel2} conjectured that any weakly octahedral coloring of $\es$ is octahedral, meaning that he does not require the coloring to be orthogonal. The full conjecture of Noel has been proven by the second author \cite{SHL}.

\subsection{Non-octahedral colorings} As we mentioned in the introduction,  we believe there exists orthogonal 4-colorings of $\es$ which are not octahedral, and our results give certain restrictions on such colorings. In contrast to the results of this paper (and as a possible candidate for a non-octahedral 4-coloring of $\es$) we would like to propose the following.

\begin{sprm*}
  What is the smallest integer $k$ for which there exists an orthogonal $k$-coloring of $\es$ where each color class is dense in $\es$ ?
\end{sprm*}

An example which shows that $k\leq 9$ was communicated to us by Jineon Baek \cite{jineon}. His argument is based on a non-Archimedean absolute value $\nu \colon \mathbb{R}\to \mathbb{R}$. (Take for instance, an extension of the 2-adic valutaion on $\mathbb{Q}$.) Here is a brief sketch of the construction, where the details are left to the reader.

For $i=1,2,3$ and $j=1,2$, define subsets
\[\def\arraystretch{1.4}
\begin{array}[h]{cccc}
X_i & \coloneqq & \{(x_1,x_2,x_3) \in \es \: : &  \nu(x_i) > \nu(x_k) \; , \; k\neq i \}, \\ 
Y_j & \coloneqq & \{(x_1,x_2,x_3) \in \es \: : &  \nu(x_j) \geq \nu(x_k) \; , \; k\neq j \}.
\end{array}
\]
Notice that the union $X_1\cup X_2\cup X_3\cup Y_1\cup Y_2$ covers $\es$. Moreover, each subset is dense in $\es$ and none of the $X_i$ contain orthogonal pairs of points. On the other hand, each $Y_j$ may contain orthogonal pairs, but for any $x\in X_j$ and $y \in Y_j$, we have $x\cdot y \neq 0$. Let $A_1$, $A_2$, $A_3$, $A_4$ be an orthogonal 4-coloring of $\es$, and for $j = 1,2$ and $k =1,2,3,4$, let \[C_{jk} = (Y_j \cap A_k) \cup (X_j \setminus A_k).\] The sets $C_{jk}$ together with $X_3$ result in a {\em covering} of $\es$ by 9 sets where each set is dense in $\es$, and no set contains a pair of orthogonal points. Finally, a simple topological argument allows us to reduce this covering to an orthogonal {\em coloring} where each color class is dense in $\es$.

\section{Acknowledgements}
We are grateful to the anonymous referee who gave several insightful remarks. We also thank Jineon Baek, Evan DeCorte, and Jon Noel for their enlightening discussions and for letting us include their examples in this paper.

\end{document}